\documentclass[journal,twoside,web]{ieeecolor}
\usepackage{generic}
\usepackage{bm}   % for bold
\usepackage{amsmath,amssymb,amsfonts}
\usepackage{indentfirst}  % for indent
\usepackage{graphicx}  % for pics
\usepackage{subfigure}
\usepackage{cite}  % for citation and reference
\usepackage{color}
\usepackage[american]{babel}
\usepackage{microtype}
\usepackage{cases}
\usepackage{enumerate}
\usepackage{hyperref}
\hypersetup{hidelinks=true}
\usepackage{textcomp}
\def\BibTeX{{\rm B\kern-.05em{\sc i\kern-.025em b}\kern-.08em
    T\kern-.1667em\lower.7ex\hbox{E}\kern-.125emX}}
\makeatletter
\@ifundefined{IEEEproof}{\newenvironment{IEEEproof}{\proof}{\endproof}}{}
\@ifundefined{ifCLASSOPTIONcaptionsoff}{\newif\ifCLASSOPTIONcaptionsoff\CLASSOPTIONcaptionsofffalse}{}

\newtheorem{theorem}{Theorem}
\newtheorem{corollary}{Corollary}
\newtheorem{lemma}{Lemma}
\newtheorem{definition}{Definition}
\newtheorem{remark}{Remark}
\newtheorem{assumption}{Assumption}

\def\mV{\mathcal{V}}
\def\mE{\mathcal{E}}

\def\mG{\mathcal{G}}

\def\mS{\mathcal{S}}
\def\mW{\mathcal{W}}
\def\mbR{\mathbb{R}}
\def\mbC{\mathbb{C}}

\def\mR{\mathcal{R}}
\def\Lg{\bm{L}}
\def\Lgz{\bm{L}_0}
\def\Lgtau{\bm{L}_{\tau}}
\def\Lgtauim{\bm{L}_{\tau}^\textup{im}}
\def\Htau{\bm{H}_{\tau}}
\def\tiLgtau{\widehat{\bm{L}}_{\tau}}
\def\tiLgz{\widehat{\bm{L}}_{0}}
\def\dmax{\overline{\lambda}}

\def\taumax{\overline{\tau}}

\def\reff{r_{\textup{eff}}}

\def\upj{\textup{j}}
\def\reff{r^\textup{eff}}
\newcommand{\re}[1]{#1^\textup{re}}

\begin{document}
\title{Consensusability of Continuous-Time Multi-Agent Systems With Unbounded Heterogeneous Constant Delays: A Signed Laplacian Perspective}

\author{Yue~Song,~\IEEEmembership{Member,~IEEE,}
          ~Mengqi~Xue,~\IEEEmembership{Member,~IEEE,}
          ~Jiazuo~Hou,~\IEEEmembership{Member,~IEEE,}
          ~Yuxi~Lu,~\IEEEmembership{Member,~IEEE,}
          ~Qi~Liu,~\IEEEmembership{Member,~IEEE}
\thanks{The authors are with the National Key Laboratory of Autonomous Intelligent Unmanned Systems, Tongji University, Shanghai 201210, China (e-mails: ysong@tongji.edu.cn, x\_starter@hotmail.com, houjz@tongji.edu.cn, yuxilu@tongji.edu.cn, qiliu@tongji.edu.cn).}
}

\markboth{}  % , ~Vol.~, No.~, ~2016}
{Song \MakeLowercase{\textit{et al.}}: }

\maketitle

\begin{abstract}
The consensus of continuous-time multi-agent systems with unbounded and heterogeneous constant delays is investigated by combining frequency-domain analysis and algebraic graph theory.
Several types of signed Laplacians are constructed to characterize consensusability under delays.
The core results are established based on the defined delay-embedded signed Laplacian, where a small-delay link creates a cooperative interaction and a possibly unbounded large-delay link creates an antagonistic interaction between the agents.
The dividing line between small and large delays is given by $\tau_{ij}=\pi/2\lambda_{\max}(\Lgz)$, where $\lambda_{\max}(\Lgz)$ refers to the maximum eigenvalue of the conventional graph Laplacian.
It is proved that the consensusability is preserved if the delay-embedded signed Laplacian is positive semi-definite with a simple zero eigenvalue.
Moreover, we derived some consensus conditions in terms of the extended effective resistance which measures the overall coupling between two sets of agents.
The obtained results provide new insights into the mechanism of delayed consensus from the interplay between the small-delay-induced cooperativeness and large-delay-induced antagonism in the underlying network topology.
\end{abstract}


% \fi
\IEEEpeerreviewmaketitle

%IEEEhowto:kopka
\section{Introduction}
Consensus is one of the most representative examples of the collective behaviors among agents in network systems.
Due to its fundamental importance, intensive attention has been paid to the mechanism and application of network consensus~\cite{olfati2007consensus,liu2026enhancing}.
Since the communication environments are imperfect and disturbances are inevitable in real-world situations, time delays commonly exist in the interactions between agents which may degrade the system performance or even destabilize the system.
Preserving consensusability (i.e., the ability of a multi-agent system to reach a consensus) under delays is a topic of both theoretical and practical value in the study of network systems.

\indent
Given a delay profile for each link in a multi-agent system, the consensusability can be computed by checking if there exist right-half-plane (RHP) roots of the system characteristic equation via various numerical methods such as Rekasius transformation \cite{munz2009stability,delice2012delay}, Pad\'{e} approximants and Chebyshev discretization \cite{milano2016small}.
In parallel, analytical studies have been conducted to reveal the mechanism of how the delays impact consensus. The frequency-domain methods usually lead to explicit conditions for consensus \cite{olfati2004consensus,bliman2008average,ma2022delay} but only apply to simplified cases, i.e., the delays are identical or the model is linear.
The time-domain methods, which are based on the Lyapunov-Krasovskii stability theory or Razumikhin stability theory, has a wider range of applicability but with more implicit result interpretations. It constructs a proper Lyapunov function/functional to quantify the dynamics and usually leads to consensus condition in the form of linear matrix inequalities (LMIs) \cite{jing2004lmi}.
We summarize the existing analytical results as follows with a classification of the delay types they handled.

\indent
1) Bounded identical delays.
For the linear consensus model, a fundamental result was obtained in \cite{olfati2004consensus} saying that the consensusability is preserved if the identical delay $\tau$ satisfies $\tau< \frac{\pi}{2\lambda_{\max}(\Lgz)}$, where $\lambda_{\max}(\Lgz)$ denotes the maximum eigenvalue of the graph Laplacian matrix $\Lgz$ (defined in the conventional sense). It concisely captures how the topological features influence the system's robustness to delays.
The later studies extend the analysis to consensus models with more general and complicated factors, such as agent's self dynamics \cite{xu2019consensusability}, switching networks \cite{lin2010consensus}, stochastic networks \cite{wang2010global}, delayed feedbacks \cite{ma2022delay}, and performance metrics of consensus dynamics over undirected graphs \cite{ghaedsharaf2019performance} and directed graphs \cite{dezfulian2022performance}.
These results also imply that a smaller $\lambda_{\max}(\Lgz)$ gives a better consensusability.

2) Bounded heterogeneous delays. It is proved in \cite{bliman2008average} that the threshold $\lambda_{\max}(\Lgz)$ also works for heterogeneous delays, i.e., the consensusability of linear networks is preserved if the delay of each link $\tau_{ij}$ satisfies $\tau_{ij}< \frac{\pi}{2\lambda_{\max}(\Lgz)}$.
In case of time-varying delays or more complicated consensus models, the problem becomes much harder.
A majority of research resorted to time-domain approaches and assumed bounded delays and bounded time derivative of delays to enable the construction of Lyapunov functions/functionals and associated LMI-based consensus conditions, e.g., \cite{zhang2018quantized,wang2020asynchronous,jiang2022multi}, to name just a few.
Usually those LMIs are feasible only when the bounds of the delays and their time derivatives are rather small.

3) Unbounded delays. Due to the difficulty caused by infinity, there are fewer works in this direction and the results are ad-hoc to specific models such as multi-agent systems without self-delays \cite{liu2010consensus}, systems under distributed delays \cite{bi2025consensus}, neural networks \cite{zhang2023master}, impulsive systems \cite{xiao2020stabilization} and Filippov systems \cite{kong2024filippov}.
Overall, there is a lack of intuitive interpretation for the role of network topology in reaching a consensus with heterogeneous and unbounded delays.

\indent
In this paper, we adopt the strategy of ``frequency-domain analysis over networks'' to characterize the heterogeneously- and unboundedly-delayed consensus in terms of the Laplacian matrices of some specially designed signed graphs carrying the delay information.
The following contributions are made.

First, we define the freqency-domain signed Laplacian which combines the strengths of both frequency-domain analysis and algebraic graph theory in characterizing system consensusability.
It is proved that the consensusability is preserved if this Laplacian matrix is positive semi-definite (PSD) with a simple zero eigenvalue at any point in the right half plane (RHP).
This result paves the way for leveraging graph theoretical tools to establish more explicit conditions for multi-agent consensus under heterogeneous delays.

Second, we construct the delay-embedded signed Laplacian which is a frequency-independent constant matrix induced from the freqency-domain signed Laplacian.
In the delay-embedded signed Laplacian, a small-delay positive-weight link ($a_{ij}>0$, $\tau_{ij} < \frac{\pi}{2\lambda_{\max}(\Lgz)}$) creates a cooperative interaction, while a large-delay link ($a_{ij}>0$, $\tau_{ij} > \frac{\pi}{2\lambda_{\max}(\Lgz)}$) or negative-weight link ($a_{ij}<0$) induces an antagonistic interaction between the agents, which quantifies the impact of heterogeneous delays in an explicit manner.
We prove that a consensus is achieved if this Laplacian is PSD with a simple zero eigenvalue.

Third, combining our previous work on the extended effective resistance \cite{song2019extension}, we derive consensus conditions in terms of the total coupling between two sets of agents considering the aggregate effects from the cooperative interactions of small-delay links and antagonistic interactions of large-delay links.
The established theory reveals how consensusability is preserved by the network topology absorbing the impact of heterogeneous delays.
Since our analysis does not rely on the assumption of bounded delays, the obtained result can handle delays of arbitrary values. In particular, it provides new insights into the case of such large delays that the traditional consensus criteria based on bounded delays fail to work.

\indent
The rest of the paper is organized as follows.
Section \ref{secformu} formulates the consensus model with unbounded heterogeneous delays.
Section \ref{secfreq} and Section \ref{secdelay} define several signed Laplacians and carry out consensus analysis.
Section \ref{seccase} illustrates the theoretical results by some numerical examples.
Section \ref{secconclu} concludes the paper.

\section{Problem formulation}\label{secformu}
\subsection{Basic notations}
Let $\mbC$, $\mbR$ denote the set of complex and real numbers, respectively.
Given a matrix $\bm{x}\in\mbC^{p\times q}$, let $\bm{x}^T$ denote its transpose and $\bm{x}^*$ its conjugate transpose.
%For an Hermitian matrix $\bm{A}$, $\bm{A}\succ0$ ($\bm{A}$ is PSD) means $\bm{A}$ is positive definite (positive semi-definite).
The notation $|\cdot|$ has different implications in different situations---$|x|$ denotes the modulus (or cardinality) if $x$ is a scalar (or a set).
$\bm{I}_p\in \mbR^{p\times p}$ denotes an identity matrix. $\bm{1}_{p}\in \mbR^{p}$ denotes a vector with all entries being one.
The italic $j$ denotes a numbering index, and the upright j denotes the square root of -1.

\subsection{Delay-embedded graph notations}
In this paper, we consider a connected weighted undirected graph to model the network with delays over which the consensus protocol will be carried out.
We denote the \textit{delay-embedded graph} $\mG$ by the tuple $\mG(\mV,\mE,\bm{A},\bm{\tau})$ where $\mV=\{1,2,...,n\}$ is the set of nodes, $\mE\subseteq\mV\times\mV$ is the set of edges, $\bm{A}=[a_{ij}]\in\mbR^{n\times n}$ ($n=|\mV|$) is the weighted adjacency matrix and $\bm{\tau}=[\tau_{ij}]\in\mbR^{n\times n}$ is the delay matrix.
An undirected edge between node $i$ and node $j$ is denoted by the unordered pair $(i,j)\in\mE$, i.e., $(i,j)\in\mE$ is equivalent to $(j,i)\in\mE$.
%The set of neighbors of node $i$ is denoted by $\mN_i=\{j|~j\in\mV,(i,j)\in\mE\}$.
The adjacency matrix $\bm{A}$ is defined such that $a_{ii}=0$, $a_{ij}=a_{ji}$ denotes the weight of edge $(i,j)\in\mE$ ($a_{ij}$ is allowed to be positive or negative), and $a_{ij}=a_{ji}=0$ if $(i,j)\notin\mE$.
%The degree of node $i$ is defined as $d_i=\sum\nolimits_{j=1}^n a_{ij}$.
The matrix $\bm{\tau}$ is defined such that $\tau_{ii}=0$, $\tau_{ij}=\tau_{ji}\geq0$ denotes the delay in edge $(i,j)\in\mE$, and $\tau_{ij}=\tau_{ji}=0$ if $(i,j)\notin\mE$.
The value of $\tau_{ij}$ is allowed to be unbounded in this paper.

\indent
The graph Laplacian matrix $\Lg(w_{ij})=[L_{ij}]\in\mbR^{n\times n}$ is defined by $L_{ii}=\sum_{j=1,j\neq i}^n w_{ij}$ and $L_{ij}=L_{ji}=-w_{ij}$, $\forall i\neq j$ where $w_{ij}$ is a certain characterization of the edge $(i,j)\in\mE$ and $w_{ij}=0$ if $(i,j)\notin\mE$.
It is free to set the values of $w_{ij}$ for different purposes, and $w_{ij}$ can be either positive or negative, where a positive (negative) $w_{ij}$ implies a cooperative (antagonistic) interaction between two nodes.
We use $\Lgz$ to denote the conventional Laplacian matrix where $w_{ij}=a_{ij}$, $\forall (i,j)\in\mE$, i.e., $\Lgz$ just reflects the topological feature of $\mG$ and the delays are not taken into consideration.
If there exist some edges $(i,j)\in\mE$ with $w_{ij}<0$, then the graph is referred to as a \textit{signed graph} and the Laplacian matrix is called a \textit{signed Laplacian} as it differs from the usual case where the off-diagonal entries are all non-positive.
In the following we will define several types of signed Laplacians tailored for consensus analysis under heterogeneous delays.

\subsection{Consensus with unbounded heterogeneous delays}
Consider a network where a group of agents interact via some links.
We assume the links have heterogeneous delays due to that the real-world links (e.g., communication links) are vulnerable to random disturbances and the delays in different links can be highly different or even unbounded.
Also some links may have zero or negligible delays if they are realized by electrical/mechanical connection or high-reliability telecommunication.
By treating the agents and links as nodes and edges, this network can be modeled by a delay-embedded graph $\mG(\mV,\mE,\bm{A},\bm{\tau})$, where $\mV$ collects the agents, $\mE$ collects the set of links, $\bm{A}$ describes the network topology and $\bm{\tau}$ collects the delay information of the links.

\indent
Then, let us consider the consensus protocol over graph $\mG(\mV,\mE,\bm{A},\bm{\tau})$
\begin{equation}\label{consensus}
\begin{split}
     \dot{x}_i(t) = \sum\nolimits_{j=1}^n a_{ij}(x_j(t-\tau_{ij}) - x_i(t-\tau_{ij})).
\end{split}
\end{equation}
Applying the Laplace transform to \eqref{consensus} gives
\begin{equation}\label{consensus-s}
\begin{split}
     s X_i(s) - x_i(0) = \sum\nolimits_{j=1}^n a_{ij}e^{-\tau_{ij}s}(X_j(s) - X_i(s))
\end{split}
\end{equation}
where $X_i(s)$ denotes the Laplace transform of $x_i(t)$.
Let $\bm{X}(s)=[X_i(s)]\in\mbC^n$ and $\bm{x}(t)=[x_i(t)]\in\mbR^n$, then \eqref{consensus-s} can be reformulated into the compact form
\begin{equation}\label{consensus-X}
\begin{split}
     \bm{X}(s) = \Htau^{-1}(s) \bm{x}(0)
\end{split}
\end{equation}
where $\Htau(s)=[H_{\tau,ij}(s)]\in\mbC^{n\times n}$ is given by
\begin{equation}\label{sLap}
\begin{split}
         H_{\tau,ij}(s) =
         \left\{
           \begin{array}{ll}
             -a_{ij}e^{-\tau_{ij}s}, &i\neq j \\
             s + \sum_{j=1,j\neq i}^n a_{ij}e^{-\tau_{ij}s}, &i=j.
           \end{array}
         \right.
\end{split}
\end{equation}

\indent
Recall that a transmission zero of a transfer matrix refers to a point where transfer matrix loses its normal rank, and a pole of a transfer matrix refers to a root of the denominator of its pole polynomial \cite{skogestad2005multivariable}.
Since $\Htau(s)$ normally has a full rank (e.g., when $s$ has a sufficiently large real part), we have the following definition for the zeros and poles of $\Htau(s)$.
\begin{definition}\label{def-pole}
    The transfer matrix $\Htau(s)$ has a zero at $s=z_0$ if $\textup{det}(\Htau(z_0))=0$, and has a pole at $s=p_0$ if $\textup{det}(\Htau(p_0))\rightarrow\infty$.
\end{definition}

\indent
Moreover, let us introduce the concept of consensus and a basic assumption below.
\begin{definition}
    For protocol \eqref{consensus}, a consensus is said to be reached if $\lim_{t\rightarrow \infty} \bm{x}(t)=c\bm{1}_n$ where $c\in\mbR$ is an arbitrary number.
    Further, an average consensus is said to be reached if $\lim_{t\rightarrow \infty} \bm{x}(t)=c\bm{1}_n$ where $c=\frac{1}{n}\bm{1}_n^T\bm{x}(0)$.
\end{definition}

\begin{assumption}\label{assump1}
   $\Lgz$ is PSD with a simple zero eigenvalue.
\end{assumption}

\indent
This assumption naturally holds in absence of negative $a_{ij}$, and we reasonably extend this property to the more generic graphs studied in this paper.
Assumption \ref{assump1} directly leads to two quick results \cite{bapat2010graphs}: 1) the eigenvector of the unique zero eigenvalue of $\Lgz$ is $\bm{1}_n$, which will be used in the following analysis; 2) system \eqref{consensus} will reach an average consensus in case of no delays, which provides a basis for the consensusability check under non-zero delays.

\indent
The existing works \cite{bliman2008average} already proved that an average consensus can still be reached under small delays bounded in terms of the largest Laplacian eigenvalue (i.e., $\tau_{ij}< \frac{\pi}{2\lambda_{\max}(\Lgz)}$).
The next section will generalize this Laplacian-based characterization of consensusability to the case of heterogeneous and possibly unbounded delays.

\begin{remark}
There are other types of delayed consensus protocol \cite{munz2012delay}, e.g., the protocol without self-delay $\dot{x}_i(t) = \sum\nolimits_{j=1}^n a_{ij}(x_j(t-\tau_{ij}) - x_i(t))$.
But this protocol only reaches a consensus, while \eqref{consensus} can possibly achieve an average consensus.
As an average consensus guarantees that the sum of the states is an invariant quantity, it is fundamentally more useful in various applications including information fusion, distributed optimization and training \cite{neal2011distributed}. So we adopt \eqref{consensus} and focus on its average consensusability.
\end{remark}

\section{Frequency-domain signed Laplacian}\label{secfreq}
For a better clarity of the following analysis, in some places we will express $s\in\mbC$ as $s=r + \upj \omega$ where $r,\omega\in\mbR$ denotes the real and imaginary part of $s$, respectively.
Let us begin with the concept of frequency-domain Laplacian matrix below.

\begin{definition}[Frequency-domain signed Laplacian]\label{def-sLap}
    Given a graph $\mG(\mV,\mE, \bm{A}, \bm{\tau})$, characterize each edge $(i,j)\in\mE$ by
    \begin{equation}\label{}
    \begin{split}
        w_{ij}(s)\big|_{s=r + \upj \omega}=a_{ij}e^{-\tau_{ij}r}\cos(\tau_{ij}\omega)
    \end{split}
    \end{equation}
    and the frequency-domain signed Laplacian $\Lgtau(s)=[L_{\tau,ij}(s)]\in\mbR^{n\times n}$ is defined by $L_{\tau,ii}(s)=\sum_{j=1,j\neq i}^n w_{ij}(s)$ and $L_{\tau,ij}(s)=L_{\tau,ji}(s)=-w_{ij}(s)$, $i\neq j$.
\end{definition}

\indent
The frequency-domain Laplacian is a signed Laplacian as it is possible to have $L_{\tau,ij}(s)>0$ at some $s\in\mbC$.
Considering that $\Htau(s)$ is complex symmetric, let us extract its real part $\re{\Htau}(s)$ by the following manipulation
\begin{equation}\label{LgtauHtau}
\begin{split}
     \re{\Htau}(s)=\frac{1}{2}[\Htau(s)+\Htau^*(s)] = r\bm{I}_n + \Lgtau(s).
\end{split}
\end{equation}
%where the superscript $*$ denotes the conjugate transpose.
It is trivial to verify that both $\re{\Htau}(s)$ and $\Lgtau(s)$ are real symmetric and $\Lgtau(0)=\re{\Htau}(0)=\Htau(0)=\Lgz$.
It will be seen later that $\re{\Htau}(s)$ serves as a key intermediate to connect $\Htau(s)$ and $\Lgtau(s)$.

\indent
We establish some lemmas below as a basis.
%\begin{assumption}\label{assump1}
%   $\Lgz$ has a simple zero eigenvalue with eigenvector $\bm{1}_n$.
%\end{assumption}

%\begin{assumption}\label{assump2}
%   There is no pole-zero cancellation in transfer matrix $\Htau(s)$.
%\end{assumption}

%\indent
%In fact, Assumption \ref{assump1} is true if $\mG$ is connected and $a_{ij}>0$, $\forall (i,j)\in\mE$. We assume this property extends to the general case with the existence of $a_{ij}<0$.
%In addition, since we focus on a network where the delays in the links are highly heterogeneous, it is unlikely to have perfect pole-zero cancellations in $\Htau(s)$ and hence Assumption \ref{assump2} is reasonable.

\begin{lemma}[\cite{chen2007final}]\label{lem-FVT}
    Let $\bm{f}(t)$ be a continuous signal which has Laplace transform $\bm{F}(s)$.
    Assume every pole of $\bm{F}(s)$ is either in the open left half plane (LHP) or at the origin, and $\bm{F}(s)$ has at most a single pole at the origin.
    Then $\lim_{t\rightarrow\infty} \bm{f}(t)=\lim_{s\rightarrow0} s\bm{F}(s)$.
\end{lemma}

\begin{lemma}\label{lem-eigen}
    %For any $s\in\mbC$, $\Htau(s)$ has a simple eigenvalue $s$ with eigenvector~$\bm{1}_n$ and has a single zero at the origin.
    $\Htau(s)$ has a single zero at the origin.
    %$\re{\Htau}(s)$ has a simple eigenvalue $\textup{Re}\{s\}$ with eigenvector~$\bm{1}_n$.
\end{lemma}

\begin{IEEEproof}
   %It is easy to verify that $\Htau(s)$ has an eigenvalue $s$ with eigenvector~$\bm{1}_n$ and has a zero at the origin.
   It is easy to verify that $\Htau(s)$ has a zero at the origin since $\Htau(0)=\Lgz$.
   Next let us prove its uniqueness by contradiction.
   %Suppose $\Htau(s)$ has $k$ ($k\geq 2$) multiple eigenvalues $s$, then $\textup{det}(\lambda\bm{I}_n-\Htau(s))$ must contain the factor $(\lambda-s)^k$.
   %At the origin ($s=0$), this factor reduces to $\lambda^k$ and it also holds $\textup{det}(\lambda\bm{I}_n-\Htau(0))=\textup{det}(\lambda\bm{I}_n-\Lgtau(0))$, which means $\textup{det}(\lambda\bm{I}_n-\Lgtau(0))$ contains the factor $\lambda^k$. Thus, $\Lgtau(0)=\Lgz$ contains multiple zero eigenvalues, which contradicts to the fact that $\Lgz$ has a simple zero eigenvalue.
   Suppose $\Htau(s)$ has $k$ ($k\geq 2$) zeros at the origin, then $\textup{det}(\Htau(s))$ takes the form $\textup{det}(\Htau(s))=s^k\cdot N(s)/D(s)$. % where $D(0)\neq 0$ due to the absence of pole-zero cancellation.
   Since $\textup{det}(\Htau(s))=\Pi_{i=1}^n \lambda_i(\Htau(s))$ and $\Htau(s)$ has an eigenvalue $s$, there exist other eigenvalues of $\Htau(s)$ whose expression contain the factor $s$ and become zero at the origin.
   It implies that $\Htau(0)=\Lgz$ has multiple zero eigenvalues, which contradicts Assumption \ref{assump1}.
\end{IEEEproof}

\begin{lemma}\label{lem-zeropole}
    If $\Htau(s)$ has a single zero at the origin and all the other zeros are in the open LHP, then system \eqref{consensus} reaches an average consensus.
\end{lemma}

\begin{IEEEproof}
   Using the property $\Htau^{-1}(s)=\frac{\textup{adj}(\Htau(s))}{\textup{det}(\Htau(s))}$, we can conclude that if $\Htau^{-1}(s)$ has a pole at $s=p_0$, then $\Htau(s)$ has a zero at $s=p_0$.
   Thus, if $\Htau(s)$ has at most a single zero at the origin and all the other zeros are in the open LHP, then every pole of $\Htau^{-1}(s)$ is either in the open LHP or at the origin, and at most one pole at the origin.
   By Lemma \ref{lem-FVT}, the dynamics of system \eqref{consensus} follow
   \begin{equation}\label{consensusFVT}
   \begin{split}
         \lim_{t\rightarrow\infty} \bm{x}(t)=\lim_{s\rightarrow0} s\Htau^{-1}(s)\bm{x}(0)=\lim_{s\rightarrow0} s\sum_{i=1}^n \frac{\bm{u}_i(s)\bm{v}_i(s)^T}{\lambda_i(s)}\bm{x}(0)
    \end{split}
    \end{equation}
    where $\lambda_i(s)$ denotes an eigenvalue of $\Htau(s)$ with the right eigenvector $\bm{u}_i(s)$ and left eigenvector $\bm{v}_i(s)$.
    It can be verified that $\Htau(s)$ always has an eigenvalue $\lambda_1=s$ with eigenvector $\bm{u}_1=\bm{v}_1=\frac{1}{\sqrt{n}}\bm{1}_n$.
    Further, since $\lim_{s\rightarrow 0}\Htau(s)=\Lgz$ contains a simple zero eigenvalue, we conclude that when $s\rightarrow0$, all the eigenvalues of $\Htau(s)$ other than $\lambda_1=s$ are nonzero.
    Thus, \eqref{consensusFVT} reduces to
    \begin{equation}
    \begin{split}
         \lim_{t\rightarrow\infty} \bm{x}(t)=\lim_{s\rightarrow0} s\frac{\bm{u}_1\bm{v}_1^T}{\lambda_1}\bm{x}(0)=\frac{1}{n}\bm{1}_n\bm{1}_n^T\bm{x}(0)
    \end{split}
    \end{equation}
    which implies an average consensus.
\end{IEEEproof}

\begin{lemma}\label{lem-Ltau}
    If $\Lgtau(s)$ is PSD with a simple zero eigenvalue for any $s\in\mS$, where $\mS$ is a subset in the open RHP, then $\Htau(s)$ does not have zeros in the set $\mS$.
    %If $\re{\Htau}(s)$ is positive definite for any $s\in\mS$, where $\mS$ is a set in the complex plane, then $\Htau(s)$ does not have zeros in the set $\mS$.
\end{lemma}

\begin{IEEEproof}
   If $\Lgtau(s)$ is PSD with a simple zero eigenvalue for any $s\in\mS$, which implies that $\bm{x}^T\Lgtau(s)\bm{x}\geq0$, $\forall \bm{x}\in\mbR^n$ and $\bm{x}^T\Lgtau(s)\bm{x}=0$ only when $\bm{x}\in\textup{span}\{\bm{1}_n\}$.
   Then, we have $\bm{x}^T\re{\Htau}(s)\bm{x}=\bm{x}^T\Lgtau(s)\bm{x}+r\|\bm{x}\|^2 >0$ holds for any $s\in\mS$ and $\bm{x}\in\mbR^n$.
   So, $\re{\Htau}(s)$ is positive definite for any $s\in\mS$.

   Next, let $\lambda_i(s)$ denote an eigenvalue of $\Htau(s)$ with the normalized eigenvector $\bm{u}_i\in\mbC^n$, i.e., $\lambda_i(s)=\bm{u}_i^*\Htau(s)\bm{u}_i$.
   Since we already proved $\re{\Htau}(s)$ is positive definite for any $s\in\mS$, which means $\bm{u}_i^*\re{\Htau}(s)\bm{u}_i>0$, then it follows  $\textup{Re}\{\lambda_i(s)\}=\frac{1}{2}\bm{u}_i^*[\Htau(s)+\Htau^*(s)]\bm{u}_i=\bm{u}_i^*\re{\Htau}(s)\bm{u}_i>0$.
   Therefore, $\textup{det}(\Htau(s))\neq 0$ for any $s\in\mS$ and hence no zeros of $\Htau(s)$ exist in $\mS$.
\end{IEEEproof}

\begin{lemma}\label{lem-Ltau-imag}
    If $\Lgtau(\upj\omega)$ is PSD with a simple zero eigenvalue for any $\omega\in\mW$, then the set of points $\{\upj\omega|~ \omega\in\mW\backslash\{0\}\}$ are not zeros of $\Htau(s)$.
\end{lemma}

\begin{IEEEproof}
   Let $\lambda_i(\upj\omega)$ denote an eigenvalue of $\Htau(\upj\omega)$ with the normalized eigenvector $\bm{u}_i\in\mbC^n$, i.e., $\lambda_i(\upj\omega)=\bm{u}_i^*\Htau(\upj\omega)\bm{u}_i$.
   If $\Lgtau(\upj\omega)=\re{\Htau}(\upj\omega)$ is PSD with only one zero eigenvalue for any $\omega\in\mW$, then the eigenvector associated with the zero eigenvalue must be $\bm{1}_n$.
   It implies that $\bm{x}^*\re{\Htau}(\upj\omega)\bm{x}\geq0$ and $\bm{x}^*\re{\Htau}(\upj\omega)\bm{x}=0$ only when $\bm{x}\in\textup{span}\{\bm{1}_n\}$.
   Thus, for any $\omega\in\mW$ we have $\textup{Re}\{\lambda_i(\upj\omega)\}=\frac{1}{2}\bm{u}_i^*[\Htau(\upj\omega)+\Htau^*(\upj\omega)]\bm{u}_i=\bm{u}_i^*\re{\Htau}(\upj\omega)\bm{u}_i>0$ if $\bm{u}_i\notin\textup{span}\{\bm{1}_n\}$, and $\textup{Re}\{\lambda_i(\upj\omega)\}=0$ if $\bm{u}_i\in\textup{span}\{\bm{1}_n\}$.
   On the other hand, it can be easily verified that $\Htau(\upj\omega)$ has an eigenvalue $\upj\omega$ with eigenvector~$\bm{1}_n$, and hence the eigenvectors of all the other eigenvalues of $\Htau(\upj\omega)$ are not in $\textup{span}\{\bm{1}_n\}$. So we conclude that $\Htau(\upj\omega)$ has an eigenvalue $\upj\omega$ and all the other eigenvalues $\lambda_i(\upj\omega)$ have positive real parts for any $\omega\in\mW$.
   Therefore, we have $\textup{det}(\Htau(\upj\omega))\neq0$ for any $\omega\in\mW\backslash\{0\}$, i.e., the zeros of $\Htau(\upj\omega)$ are not in $\mW\backslash\{0\}$.
\end{IEEEproof}

\indent
With the above lemmas, we now arrive at a neat statement on the system consensus in terms of the frequency-domain signed Laplacian.

\begin{theorem}\label{thm-Ltau}
     Given the delay-embedded graph $\mG(\mV,\mE,\bm{A},\bm{\tau})$, system \eqref{consensus} reaches an average consensus if the frequency-domain signed Laplacian $\Lgtau(s)$ is PSD with a simple zero eigenvalue for any $s$ in the RHP.
\end{theorem}

\begin{IEEEproof}
    It is a direct result from Lemmas~\ref{lem-eigen}-\ref{lem-Ltau-imag} by setting $\mS$ as the open RHP and $\mW=\mbR$.
\end{IEEEproof}

\indent
Theorem \ref{thm-Ltau} characterizes the consensusability under heterogeneous delays in terms of $\Lgtau(s)$ which is defined over the network topology and delay profile at all kinds of frequencies.
It implies that $\Lgtau(s)$ serves as a ``network-oriented'' transfer matrix that plays a similar role to $\Htau(s)$ in checking consensus.
By Theorem \ref{thm-Ltau}, the consensus problem is transformed from checking the frequency response of $\Htau(s)$ to checking the spectrum of $\Lgtau(s)$ which facilitates the application of graph theoretic tools.

\indent
Despite its theoretical insights, Theorem \ref{thm-Ltau} does not have high practicality as it is impossible to check $\Lgtau(s)$ at every point $s$ in the RHP.
In the next section we will construct another type of signed Laplacian matrix which is independent of $s$ and reflects the impact of delays in a more explicit manner.

\begin{remark}
We will define totally three types of signed Laplacians in this paper.
In addition to $\bm{L}_{\tau}(s)$, the following sections will further introduce $\tiLgtau$ (see Definition \ref{def-tauLap}) and $\tiLgz$ (see \eqref{sLap}), which are rooted from $\bm{L}_{\tau}(s)$.
$\tiLgtau$ captures the worst case of $\bm{L}_{\tau}(s)$ for $s$ in the RHP, and $\tiLgz$ is a further simplification of $\tiLgtau$.
All of them will be useful to characterize consensusability in the respective situations.
\end{remark}

\section{Frequency-independent signed Laplacian}\label{secdelay}
\subsection{Analysis based on delay-embedded signed Laplacian}
Let us define the following subset in the RHP
\begin{equation}
\begin{split}
     \mR= \{s=r+\upj\omega|~0\leq r \leq \dmax,~|\omega| \leq \dmax \}
\end{split}
\end{equation}
where, for simplicity of result presentation, we define
\begin{equation}
\begin{split}
     \dmax  &\triangleq \lambda_{\max}(\Lgz^\textup{abs})
\end{split}
\end{equation}
and $\Lgz^\textup{abs}$ denotes the Laplacian matrix obtained by setting $w_{ij}=|a_{ij}|$, $\forall i\neq j$. Note that $\Lgz^\textup{abs}=\Lgz$ if $a_{ij}\geq 0$, $\forall i\neq j$.
The following analysis shows that the RHP points outside the set $\mR$ are not of interest to consensus analysis.

\begin{lemma}\label{lem-mono}
    Consider a Laplacian matrix $\Lg=[L_{ij}]\in\mbR^{n\times n}$ defined by $L_{ii}=\sum_{j=1,j\neq i}^n w_{ij}$ and $L_{ij}=L_{ji}=-w_{ij}$, $\forall i\neq j$.
    Each eigenvalue of $\Lg$ monotonically increases with the value of $w_{ij}$ of each edge.
\end{lemma}

\begin{IEEEproof}
    Let $\lambda_k$ denote a nonzero eigenvalue of a general-sense Laplacian $\bm{L}$ with normalized eigenvector $\bm{u}_k=[u_{ki}]\in\mbR^n$.
    By eigen-sensitivity analysis \cite{petersen2008matrix}, we have $\frac{\partial \lambda_k}{\partial w_{ij}} = (u_{ki}-u_{kj})^2\geq 0$, which completes the proof.
\end{IEEEproof}

\begin{lemma}\label{lem-nonsingular}
    If at least one of the real symmetric matrices $\bm{A},\bm{B}\in\mbR^{n\times n}$ is positive definite or negative definite, then $\bm{A}+\upj\bm{B}$ is nonsingular.
\end{lemma}

\begin{IEEEproof}
    Let us prove it by contradiction.
    Suppose $\bm{A}+\upj\bm{B}$ is singular, then there exists a non-zero vector $\bm{x}$ such that $(\bm{A}+\upj\bm{B})\bm{x}=\bm{0}$, which further leads to
    \begin{equation}\label{nonsingular}
    \begin{split}
        \bm{x}^*\bm{A}\bm{x}=-\upj \bm{x}^*\bm{B}\bm{x}.
    \end{split}
    \end{equation}
    Since $\bm{x}^*\bm{A}\bm{x}$ and $\bm{x}^*\bm{B}\bm{x}$ both take real values, \eqref{nonsingular} enforces that $\bm{x}^*\bm{A}\bm{x}=\bm{x}^*\bm{B}\bm{x}=0$.
    It implies that $\bm{A}$ and $\bm{B}$ are not positive definite or negative definite, which yields a contradiction.
\end{IEEEproof}

\begin{lemma}\label{lem-regionR}
    A RHP point $s=r+\upj\omega$ can be a zero of $\Htau(s)$ only if $s\in\mR$.
\end{lemma}

\begin{IEEEproof}
    It suffices to show that $\Htau(s)$ is nonsingular when $s$ is a RHP point locating outside $\mR$.
    First, $\Htau(s)$ can be expressed as
    \begin{equation}
    \begin{split}
        \Htau(s)=r\bm{I}_n + \Lgtau(s) + \upj(\omega\bm{I}_n -  \Lgtauim(s))
    \end{split}
    \end{equation}
    where $\Lgtauim(s)=[L_{\tau,ij}^{\textup{im}}(s)]\in\mbR^{n\times n}$ is a Laplacian matrix defined by $L_{\tau,ij}^{\textup{im}}(s)=-a_{ij}e^{-\tau_{ij}r}\sin(\tau_{ij}\omega)$, $i\neq j$ and $L_{\tau,ii}^{\textup{im}}(s)=-\sum_{j=1,j\neq i}^n L_{\tau,ij}^{\textup{im}}$.
    Thus, the characterization of each edge $(i,j)$ in $\Lgtau(s)$ and $\Lgtauim(s)$ is bounded between $-|a_{ij}|$ and $|a_{ij}|$.
    If $s$ locates outside $\mR$ (i.e., $r> \dmax$ or $|\omega|> \dmax$), by Lemma~\ref{lem-mono} we have $-\dmax\leq\lambda_k(\Lgtau(s))\leq \dmax$ and $-\dmax\leq \lambda_k(\Lgtauim(s))\leq \dmax$, $k=1,...,n$.
    Then, it is trivial to conclude that $r\bm{I}_n + \Lgtau(s)$ is positive definite if $r> \dmax$,
    $\omega\bm{I}_n -  \Lgtauim(s)$ is positive definite if $\omega> \dmax$, and $\omega\bm{I}_n -  \Lgtauim(s)$ is negative definite if $\omega < -\dmax$.
    Let $\bm{A}=r\bm{I}_n + \Lgtau(s)$ and $\bm{B}=\omega\bm{I}_n -  \Lgtauim(s)$.
    By Lemma \ref{lem-nonsingular} we have $\Htau(s)=\bm{A}+\upj \bm{B}$ is nonsingular in case of $r> \dmax$ or $|\omega|> \dmax$.
    Therefore, $s$ cannot be a zero of $\Htau(s)$ at a RHP point locating outside $\mR$.
\end{IEEEproof}

\begin{corollary}\label{cor-Ltau}
    Given the delay-embedded graph $\mG(\mV,\mE,\bm{A},\bm{\tau})$, system \eqref{consensus} reaches an average consensus if $\Lgtau(s)$ is PSD with a simple zero eigenvalue for any $s\in\mR$.
\end{corollary}

\begin{IEEEproof}
    It follows from Theorem \ref{thm-Ltau} and Lemma \ref{lem-regionR}.
\end{IEEEproof}

\indent
By Corollary \ref{cor-Ltau}, the region to check the spectrum of $\Lgtau(s)$ is reduced from the whole RHP to the set $\mR$.
Although $\mR$ still contain an infinite number of points, the boundedness of $\mR$ enables us to figure out the worst case and further reduce the number of checkpoints from infinite to finite.
Based on this idea, we construct the following signed Laplacian matrix that captures the worst case in $\mR$ and hence get rids of the frequency domain.

\begin{definition}[Delay-embedded signed Laplacian]\label{def-tauLap}
    Given a graph $\mG(\mV,\mE, \bm{A}, \bm{\tau})$, characterize each edge $(i,j)\in\mE$ by the following rules
    \begin{subequations}\label{tauLap}
    \begin{align}
         &w_{ij} = a_{ij},~\textup{if}~a_{ij}>0,~\tau_{ij}=0 \label{tauLap1}\\
         &w_{ij} = a_{ij}e^{-\tau_{ij}\dmax}\cos(\tau_{ij}\dmax),~\textup{if}~a_{ij}>0,~0< \tau_{ij}\dmax < \frac{\pi}{2} \label{tauLap2}\\
         &w_{ij} = a_{ij}\cos(\tau_{ij}\dmax),~\textup{if}~a_{ij}>0,~\frac{\pi}{2} \leq \tau_{ij}\dmax < \pi \label{tauLap3}\\
         &w_{ij} = -a_{ij},~\textup{if}~a_{ij}>0,~\tau_{ij}\dmax\geq \pi \label{tauLap4}\\
         &w_{ij} = a_{ij},~\textup{if}~a_{ij}<0 \label{tauLap5}
    \end{align}
    \end{subequations}
     and the delay-embedded signed Laplacian $\tiLgtau=[\widehat{L}_{\tau,ij}]\in\mbR^{n\times n}$ is defined by $\widehat{L}_{\tau,ii}=\sum_{j=1,j\neq i}^n w_{ij}$ and $\widehat{L}_{\tau,ij}=\widehat{L}_{\tau,ji}=-w_{ij}$, $i\neq j$.
\end{definition}

\begin{theorem}\label{thm-tiLtau}
    Given the delay-embedded graph $\mG(\mV,\mE,\bm{A},\bm{\tau})$, system \eqref{consensus} reaches an average consensus if the delay-embedded signed Laplacian $\tiLgtau$ is PSD with a simple zero eigenvalue.
\end{theorem}

\begin{IEEEproof}
    By Lemma \ref{lem-mono}, if we construct a constant Laplacian matrix, say $\tiLgtau$, where $w_{ij}$ takes the minimum value of $a_{ij}e^{-\tau_{ij}r}\cos(\tau_{ij}\omega)$ in the set $s\in\mR$, then it suffices to guarantee the positive semi-definiteness of $\Lgtau(s)$ by the positive semi-definiteness of $\tiLgtau$.
    In the following let us explain why $w_{ij}$ is defined by \eqref{tauLap}.

    Case 1: for edge $(i,j)$ with $a_{ij}>0$ and $\tau_{ij}=0$, we have $w_{ij}=\min_{s\in\mR} a_{ij}e^{-\tau_{ij}r}\cos(\tau_{ij}\omega)=a_{ij}$, leading to \eqref{tauLap1}.

    Case 2: for edge $(i,j)$ with $a_{ij}>0$ and $0< \tau_{ij}\dmax < \frac{\pi}{2}$, it follows that $|\tau_{ij}\omega|\leq \tau_{ij}\dmax<\frac{\pi}{2}$. So we have $w_{ij}=\min_{s\in\mR} a_{ij}e^{-\tau_{ij}r}\cos(\tau_{ij}\omega)=a_{ij}e^{-\tau_{ij}\dmax}\cos(\tau_{ij}\dmax)>0$, which is achieved at $r=\omega=\dmax$ and leads to \eqref{tauLap2}.

    Case 3: for edge $(i,j)$ with $a_{ij}>0$ and $\frac{\pi}{2} \leq \tau_{ij}\dmax < \pi$, it follows that $|\tau_{ij}\omega|\leq \tau_{ij}\dmax<\pi$. So we have $w_{ij}=\min_{s\in\mR} a_{ij}e^{-\tau_{ij}r}\cos(\tau_{ij}\omega)=a_{ij}\cos(\tau_{ij}\dmax)\leq 0$, which is achieved at $r=0,\omega=\dmax$ and leads to \eqref{tauLap3}.

    Case 4: for edge $(i,j)$ with $a_{ij}>0$ and $\tau_{ij}\dmax\geq \pi$, it follows that $|\tau_{ij}\omega|$ can take value greater than or equal to $\pi$. So we have $w_{ij}=\min_{s\in\mR} a_{ij}e^{-\tau_{ij}r}\cos(\tau_{ij}\omega)=-a_{ij}< 0$, which is achieved at $r=0,\omega=\pi/\tau_{ij}$ and leads to \eqref{tauLap4}.

    Case 5: for edge $(i,j)$ with $a_{ij}<0$, we have $w_{ij}=\min_{s\in\mR} a_{ij}e^{-\tau_{ij}r}\cos(\tau_{ij}\omega)=a_{ij}<0$, which is achieved at $r=\omega=0$ and leads to \eqref{tauLap5}.

    Constructing $\tiLgtau$ by the above rules, we conclude that if $\tiLgtau$ is PSD with a simple zero eigenvalue, then $\Lgtau(s)$ is PSD with a simple zero eigenvalue for any $s\in\mR$, and hence \eqref{consensus} reaches an average consensus by Corollary \ref{cor-Ltau}.
\end{IEEEproof}

Theorem \ref{thm-tiLtau} leads to the following classification of delays, from which we have important observations on how the delays impact consensusability:
\begin{enumerate}
      \item A zero-delay positive-weight link $(i,j)\in\mE$ makes a positive contribution to consensus as the corresponding $w_{ij}=a_{ij}>0$ (see \eqref{tauLap1}) induces a cooperative interaction between agent $i$ and agent $j$.
      \item A small-delay positive-weight link (i.e., $\tau_{ij}\dmax < \frac{\pi}{2}$) still induces a cooperative interaction, which will not cause the loss of consensusability but may deteriorate the system performance.
          By \eqref{tauLap2}, we have $0<w_{ij}<a_{ij}$, which means the link still makes a positive contribution to consensus but the contribution is weaker than the zero-delay case.
      \item A (possibly unbounded) large-delay positive-weight link (i.e., $\tau_{ij}\dmax \geq \frac{\pi}{2}$) flips the actual interaction from cooperative to antagonistic. It gives rise to the ``large-delay-induced antagonism'' that can cause a loss of consensusability.
          This can be seen from \eqref{tauLap3} and \eqref{tauLap4} where we have $w_{ij}<0$ that makes a negative contribution to consensus.
      \item A negative-weight link $(i,j)\in\mE$, regardless of its delay, makes a negative contribution to consensus as $w_{ij}=a_{ij}<0$ (see \eqref{tauLap5}) induces an antagonistic interaction.
\end{enumerate}

\begin{remark}
   Consider a link with an infinite delay $\tau_{ij}\to \infty$, which is unrealistic but mathematically notable.
   In this extreme case, agents $i$ and $j$ have no actual interaction in any finite time horizon, and hence link $(i,j)$ is practically absent ($w_{ij}=0$).
   On the other hand, the rule in \eqref{tauLap4} gives a conservative assessment by assigning $w_{ij}=-a_{ij}$. This is due to that $w_{ij}$ captures the worst-case value of $e^{-\tau_{ij}r}\cos(\tau_{ij}\omega)$ and
   $w_{ij}= \min_{s\in\mR}a_{ij}e^{-\tau_{ij}r}\cos(\tau_{ij}\omega)=-a_{ij}$ holds true for any finite delay $\tau_{ij}$.
   Although Definition \ref{def-tauLap} is compatible with infinite delays, the major focus here is to analyze consensusability under large finite delays (that make traditional consensus criteria invalid).
\end{remark}

\begin{remark}
   It has been reported that an intentionally designed large-delay link may have positive effect on consensusability in some special cases \cite{ma2023intentional}.
   As a limitation of Theorem \ref{thm-tiLtau}, it is not applicable to the design of intentional delays.
   The essence of Theorem \ref{thm-tiLtau} is to check the system robustness to the worst-case effect in frequency domain caused by the delayed links, which misses to catch the potential benefit from intentional delays.
\end{remark}

\begin{remark}
   For the delayed consensus protocol \eqref{consensus} over a directed graph with asymmetric edge weights and delays, we can still use \eqref{LgtauHtau} to define $\Lgtau(s)= \frac{1}{2}(\Htau(s)+\Htau^*(s))-r\bm{I}$, which is a complex-valued Hermitian matrix.
   In this case, Theorem \ref{thm-Ltau} still holds since the quadratic form analysis in the proof is valid for Hermitian matrices.
   However, $\bm{L}_{\tau}(s)$ refers to the Laplacian matrix with complex-valued $w_{ij}$, which is beyond the scope of signed graphs that are designed for real-valued $w_{ij}$.
   Meanwhile, the construction of $\tiLgtau$, i.e., a worst-case characterization of $\bm{L}_{\tau}(s)$ for $s\in\mR$, needs to be reconsidered.
   Consequently, Theorem \ref{thm-tiLtau} is not directly applicable to directed graphs.
   An extension to directed graphs will be a future direction.
\end{remark}

\indent
Based on Theorem \ref{thm-tiLtau}, we derive some intuitive corollaries as follows.

\begin{corollary}\label{cor-taumax}
    Given the delay-embedded graph $\mG(\mV,\mE,\bm{A},\bm{\tau})$ with $a_{ij}\geq 0$, $\forall i\neq j$, system \eqref{consensus} reaches an average consensus if $\taumax < \pi/2\dmax$, where $\taumax=\max_{(i,j)\in\mE}\tau_{ij}$ denotes the maximum delay among all the links in the network.
\end{corollary}

\begin{IEEEproof}
    If $\taumax < \pi/2\dmax$, then the Laplacian matrix $\tiLgtau$ describes a connected graph with $w_{ij}>0$, $\forall (i,j)\in\mE$. Thus, we have $\tiLgtau$ is PSD with a simple zero eigenvalue and the proof is done with Theorem~\ref{thm-tiLtau}.
\end{IEEEproof}

\indent
The statement of Corollary~\ref{cor-taumax} is equivalent to [\citen{bliman2008average}, Theorem 5] and reduces to [\citen{olfati2004consensus}, Theorem 10] under identical delays.
So Theorem \ref{thm-tiLtau} includes those classical results as a special case.
Moreover, the proof of Corollary~\ref{cor-taumax} provides a new interpretation for consensus under small delays, i.e., those small-delay links do not introduce antagonistic interactions into the system.

\begin{corollary}\label{cor-tau-ind}
    Given the delay-embedded graph $\mG(\mV,\mE,\bm{A},\bm{\tau})$, characterize each edge $(i,j)\in\mE$ by
    \begin{equation}\label{sLap}
    \begin{split}
         w_{ij} =
         \left\{
           \begin{array}{ll}
             a_{ij}, &a_{ij}>0,~\tau_{ij}=0 \\
             -a_{ij}, &a_{ij}>0,~\tau_{ij}>0\\
             a_{ij}, &a_{ij}<0
           \end{array}
         \right.
    \end{split}
    \end{equation}
    and define the delay-independent signed Laplacian $\tiLgz$ by $\widehat{L}_{0,ii}=\sum_{j=1,j\neq i}^n w_{ij}$ and $\widehat{L}_{0,ij}=\widehat{L}_{\tau,ji}=-w_{ij}$, $i\neq j$.
    If $\tiLgz$ is PSD with a simple zero eigenvalue, then system \eqref{consensus} reaches an average consensus regardless of the values of those nonzero delays $\tau_{ij}>0$.
\end{corollary}

\begin{IEEEproof}
    It can be easily verified that $\tiLgz$ is ``worst of the worst case'' of $\tiLgtau$ and $\tiLgtau - \tiLgz$ is PSD, and hence $\tiLgtau$ is PSD with a simple zero eigenvalue if $\tiLgz$ is PSD with a simple zero eigenvalue.
    Thus, \eqref{consensus} reaches an average consensus by Theorem \ref{thm-tiLtau}.
\end{IEEEproof}

\indent
Corollary~\ref{cor-tau-ind} provides a delay-independent condition for network consensus. It shows that the system can endure even unbounded delays if the network topology alone is sufficiently strong to ``neutralize'' the worst possible antagonistic effects caused by those large-delay links.
In addition, $\tiLgz$ being indefinite does not necessarily mean a failure of consensus since Corollary~\ref{cor-tau-ind} is a sufficient condition.

\subsection{Analysis based on effective resistance}
The effective resistance is an important graph theoretic concept originated from electrical networks.
Consider an electrical network described by the Laplacian matrix $\bm{L}(w_{ij})$, i.e., it has the same topology as the underlying graph and the resistance of each electric line $(i,j)$ is equal to $w_{ij}^{-1}$.
In this network, suppose node $i$ connects a source with a unit current injection, node $j$ connects a source with a unit current ejection and the other nodes keep open circuit.
Then, the effective resistance between node $i$ and $j$ equals the voltage difference between node $i$ and $j$ \cite{klein1993resistance}, which characterizes the overall coupling between a pair of nodes (see Fig. \ref{fig-reff} for illustration).
In our previous work \cite{song2019extension}, an extended definition of effective resistance is proposed to characterize the overall coupling between two sets of nodes in a network.
It is revealed that the sign of effective resistance is closely linked to the definiteness of Laplacian matrices \cite{chen2016characterizing,zelazo2017robustness,song2019extension}.
Thus, using the effective resistance notion over the delay-embedded signed Laplacian, this subsection further investigates the impact of heterogeneous delays in terms of the coupling between agents.

\begin{figure}[!h]
  \centering
  \includegraphics[width=3.5in]{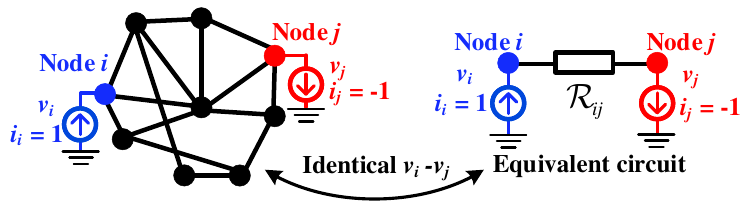}
  \caption{Circuit interpretation of effective resistance.}
  \label{fig-reff}
\end{figure}

\indent
Let us introduce some additional notations for the convenience of presenting our results.
Given two non-empty disjoint sets of nodes in a graph, say $\mV_a, \mV_b\subset \mV$, and let $\mV_c=\mV\backslash (\mV_a\cup\mV_b)$ be the set of remaining nodes. For simplicity, denote $a=|\mV_a|$, $b=|\mV_b|$ and $c=|\mV_c|$.
Then the Laplacian matrix $\Lg$ can be partitioned as
\begin{equation}
\begin{split}
   \Lg=
   \begin{bmatrix}
     \bm{L}_{aa}   &   \bm{L}_{ab}   &   \bm{L}_{ac}     \\
     \bm{L}_{ab}^T   &   \bm{L}_{bb}   &   \bm{L}_{bc}    \\
      \bm{L}_{ac}^T  &   \bm{L}_{bc}^T   &   \bm{L}_{cc}    \\
   \end{bmatrix}
\end{split}
\end{equation}
and the extended effective resistance is defined below.

\begin{definition}[Extended effective resistance \cite{song2019extension}]\label{defreff}
   Let $\mV_a, \mV_b\subset \mV$ be two non-empty disjoint sets of nodes in the graph $\mG(\mV,\mE,\bm{A})$, and $\mV_c=\mV\backslash (\mV_a\cup\mV_b)$ be the set of remaining nodes.
   Let $\Lg$ be a Laplacian matrix defined over $\mG$ with a certain characterization of edges.
   The extended effective resistance between $\mV_a$ and $\mV_b$ is defined as
   \begin{equation}
   \begin{split}
      \reff(\Lg, \mV_a, \mV_b)&=[\bm{e}_{\mV_a}^T(\Lg/ \bm{L}_{cc})\bm{e}_{\mV_a}]^{-1} \label{equreff}
   \end{split}
   \end{equation}
   where the Schur complement takes the form
   \begin{equation}
   \begin{split}
      \Lg/ \bm{L}_{cc} =
   \begin{bmatrix}
     \bm{L}_{aa}   &   \bm{L}_{ab}      \\
     \bm{L}_{ab}^T   &   \bm{L}_{bb}       \\
   \end{bmatrix}-
   \begin{bmatrix}
    \bm{L}_{ac}     \\
    \bm{L}_{bc}    \\
   \end{bmatrix}
   \bm{L}_{cc}^{-1}
   \begin{bmatrix}
     \bm{L}_{bc}^T   &   \bm{L}_{cc}    \\
   \end{bmatrix}
   \end{split}
   \end{equation}
   and $\bm{e}_{\mV_a}\in\mbR^{a+b}$ denotes a vector with those entries indexed by the set $\mV_a$ being one and the other entries being zero.
\end{definition}

\indent
It is proved that \cite{song2019extension}
\begin{equation}\label{reffcompat}
\begin{split}
   \reff(\Lg, \{i\}, \{j\}) = (\bm{e}_i-\bm{e}_j)^T\Lg^{\dag}(\bm{e}_i-\bm{e}_j)
\end{split}
\end{equation}
where $\bm{e}_i\in\mbR^n$ is a vector with its $i$-th entry being one and all the other entries being zero, and $\Lg^{\dag}$ denotes the Moore-Penrose inverse of $\Lg$.
The right hand side of \eqref{reffcompat} is the formula for the conventional effective resistance between node $i$ and $j$, i.e., the extended definition is downward compatible.

\indent
Now we arrive at the following results based on the extended effective resistance.

\begin{theorem}\label{thm-reff}
    Given the delay-embedded graph $\mG(\mV,\mE,\bm{A},\bm{\tau})$, system \eqref{consensus} reaches an average consensus if $\reff(\tiLgtau, \mV_a, \mV_b)> 0$ for any two non-empty disjoint sets $\mV_a, \mV_b\subset \mV$.
\end{theorem}

\begin{IEEEproof}
    By [\citen{song2019extension},~Corollary~2], $\tiLgtau$ is PSD and $\tiLgtau$ is PSD with a simple zero eigenvalue if and only if $\reff(\tiLgtau, \mV_a, \mV_b)> 0$ holds for any two non-empty disjoint sets $\mV_a, \mV_b$. Then, applying Theorem~\ref{thm-tiLtau} guarantees that \eqref{consensus} reaches an average consensus.
\end{IEEEproof}

\indent
A positive effective resistance implies the normal situation with passivity and cooperative coupling between the agents, while a negative one implies the loss of passivity and overall antagonistic coupling between the agents.
Thus, Theorem \ref{thm-reff} has the following interpretation.
If the cooperative interactions (from zero-delay or small-delay links) prevail over those antagonistic interactions (from large-delay links) so that the overall coupling between any two sets of agents still present cooperativeness, then the consensusability is preserved.
Also note that Theorem \ref{thm-reff} still holds true when the signed Laplacian $\tiLgtau$ in its statement is replaced by $\tiLgz$, but it leads to a more conservative assessment of the impact of delays.

\begin{corollary}\label{cor-reff}
    Given the delay-embedded graph $\mG(\mV,\mE,\bm{A},\bm{\tau})$.
    Assume there is an edge $(k,l)\in\mE$ with $\tau_{kl}\dmax > \frac{\pi}{2}$, and let $\tiLgtau^{(k,l)}$ denote the delay-embedded Laplacian after deleting edge $(k,l)$.
    Suppose $\tiLgtau^{(k,l)}$ is PSD with a simple zero eigenvalue, then \eqref{consensus} reaches an average consensus if $\reff(\tiLgtau, \{k\}, \{l\})>0$, or equivalently, the following inequality holds
   \begin{subequations}\label{singlenegative}
   \begin{align}
      &-a_{kl}\cos(\tau_{kl}\dmax) < 1/\reff(\tiLgtau^{(k,l)},\{k\},\{l\}),\textup{if}~\frac{\pi}{2} < \tau_{kl}\dmax < \pi \\
      &a_{kl}< 1/\reff(\tiLgtau^{(k,l)},\{k\},\{l\}),\textup{if}~\tau_{kl}\dmax \geq \pi.
   \end{align}
   \end{subequations}
\end{corollary}

\begin{IEEEproof}
    First, let us prove that $\reff(\tiLgtau, \{k\}, \{l\})>0$ is equivalent to \eqref{singlenegative}.
    By Definition \ref{def-tauLap}, it holds that
    \begin{equation}\label{Lgtaukl}
    \begin{split}
         \tiLgtau = \tiLgtau^{(k,l)} +\bm{e}_{kl}^T w_{kl}\bm{e}_{kl}.
    \end{split}
    \end{equation}
    where $\bm{e}_{kl}=\bm{e}_k-\bm{e}_l$.
    Since $\tiLgtau^{(k,l)}$ is PSD with a simple zero eigenvalue, the generalized Sherman-Morrison-Woodbury formula \cite{jiao2025unified} applies to the Moore-Penrose inverse of \eqref{Lgtaukl}, which gives
    \begin{equation}\label{SMW}
    \begin{split}
         \tiLgtau^{\dag} = (\tiLgtau^{(k,l)})^{\dag} - y^{-1}(\tiLgtau^{(k,l)})^{\dag}\bm{e}_{kl}^T\bm{e}_{kl}(\tiLgtau^{(k,l)})^{\dag}
    \end{split}
    \end{equation}
    where
    \begin{equation}\label{}
    \begin{split}
         y=w_{kl}^{-1}+\bm{e}_{kl}^T(\tiLgtau^{(k,l)})^{\dag}\bm{e}_{kl}
         =w_{kl}^{-1}+\reff(\tiLgtau^{(k,l)},\{k\},\{l\}).
    \end{split}
    \end{equation}
    Substituting \eqref{SMW} into \eqref{reffcompat}, we obtain the following expression for $\reff(\tiLgtau, \{k\}, \{l\})$
    \begin{equation}\label{SMW2}
    \begin{split}
         &\reff(\tiLgtau, \{k\}, \{l\}) = \bm{e}_{kl}^T\Lg^{\dag}\bm{e}_{kl}\\
         &= \reff(\tiLgtau^{(k,l)},\{k\},\{l\}) - [\reff(\tiLgtau^{(k,l)},\{k\},\{l\})]^2y^{-1} \\
         &=\frac{w_{kl}^{-1}\cdot\reff(\tiLgtau^{(k,l)},\{k\},\{l\})}{w_{kl}^{-1}+\reff(\tiLgtau^{(k,l)},\{k\},\{l\})}.
    \end{split}
    \end{equation}
    Since $w_{kl}<0$ and $\reff(\tiLgtau^{(k,l)},\{k\},\{l\})>0$ (due to $\tiLgtau^{(k,l)}$ is PSD), \eqref{SMW2} implies $\reff(\tiLgtau, \{k\}, \{l\})>0$ if and only if $w_{kl}^{-1}+\reff(\tiLgtau^{(k,l)},\{k\},\{l\})<0$ which is equivalent to \eqref{singlenegative}.

    Next, we prove that $\reff(\tiLgtau, \{k\}, \{l\})>0$ ensures $\tiLgtau$ being PSD with a simple zero eigenvalue. Construct the matrix
    \begin{equation}\label{}
    \begin{split}
         \bm{X} = \begin{bmatrix}
                    \tiLgtau^{(k,l)} & \bm{e}_{kl} \\
                    \bm{e}_{kl}^T & |w_{kl}^{-1}| \\
                  \end{bmatrix}
    \end{split}
    \end{equation}
    and it follows from [\citen{ouellette1981schur}, Theorem 4.7] that
    \begin{equation}\label{inertia}
    \begin{split}
         \textup{In}(\bm{X})&=\textup{In}(\tiLgtau^{(k,l)}) + \textup{In}(\bm{X}/\tiLgtau^{(k,l)})\\
         &= \textup{In}(|w_{kl}^{-1}|) + \textup{In}(\bm{X}/|w_{kl}^{-1}|)
    \end{split}
    \end{equation}
    where $\textup{In}()=(i_+, i_-, i_0)$ denotes the matrix inertia with $i_+, i_-, i_0$ being the number of positive, negative and zero eigenvalues, respectively; and the Schur complements are
    \begin{equation}\label{}
    \begin{split}
    \bm{X}/\tiLgtau^{(k,l)}&=|w_{kl}^{-1}|-\bm{e}_{kl}^T(\tiLgtau^{(k,l)})^{\dag}\bm{e}_{kl}=-y\\
    \bm{X}/|w_{kl}^{-1}|&=\tiLgtau^{(k,l)}-\bm{e}_{kl}|w_{kl}^{-1}|\bm{e}_{kl}^T=\tiLgtau.
    \end{split}
    \end{equation}
    If $\reff(\tiLgtau, \{k\}, \{l\})>0$, then we have $y<0$ and hence $\textup{In}(\bm{X}/\tiLgtau^{(k,l)})=\textup{In}(-y)=(1,0,0)$.
    Also it is known that $\textup{In}(\tiLgtau^{(k,l)})=(n-1,0,1)$ and $\textup{In}(|w_{kl}^{-1}|)=(1,0,0)$.
    Then, by \eqref{inertia} we have $\textup{In}(\bm{X}/|w_{kl}^{-1}|)=\textup{In}(\tiLgtau)=(n-1,0,1)$, i.e., $\tiLgtau$ is PSD with a simple zero eigenvalue. Therefore, system \eqref{consensus} reaches an average consensus by Theorem \ref{thm-tiLtau}.
\end{IEEEproof}

\indent
Focusing on the antagonistic interaction induced by one particular large-delay link $(k,l)$, Corollary \ref{cor-reff} estimates an upper bound for the delay severity of this link that will not impact consensusability.
This upper bound turns out to be $1/\reff(\tiLgtau^{(k,l)},\{k\},\{l\})$, i.e., the ``effective conductance'' (as a measure of coupling strength) between agent $k$ and agent $l$ in the remaining network excluding the link $(k,l)$.
It implies that the system can tolerate a larger delay in the link $(k,l)$ if the remaining network has reduced delays or increased values of edge weights.
The case study will illustrate how Corollary \ref{cor-reff} can be used to guide the placement of new links with delays.

\section{Numerical examples}\label{seccase}
\subsection{Illustration of theoretical results}
The obtained results are illustrated by numerical experiments on the 14-agent network shown in Fig. \ref{fig-case14}.
This network originates from the IEEE 14-bus system which is a benchmark system in power grids.

\subsubsection{single delayed link}
Suppose all the links in the 14-agent network are unweighted ($a_{ij}=1$), the link $(2,5)$ has a delay and all the other links have zero delays.
In this case, a direct calculation gives that the delay-independent signed Laplacian $\tiLgz$ is PSD with a simple zero eigenvalue.
By Corollary \ref{cor-tau-ind}, the network consensus is robust to an unbounded delay in the link $(2,5)$. This conclusion is illustrated by Fig.~\ref{fig-delay1000} showing a consensus is achieved when the link $(2,5)$ has a delay of 1000 seconds, despite its slow convergence\footnote{The initial values $\bm{x}(0)$ for the consensus simulations in Fig.~\ref{fig-delay} and Fig.~\ref{fignewlink} are randomly chosen from the interval $[0,1]$.}.

\subsubsection{multiple delayed links}
Suppose three links $(2,5)$, $(6,13)$, $(2,4)$ have delays of $\tau_{2,5}=0.14$, $\tau_{6,13}=0.28$ and $\tau_{2,4}=0.42$.
The links $(6,13)$ and $(7,9)$ induce antagonistic interactions $w_{6,13}=-0.2421$ and $w_{2,4}=-0.9136$, while the link $(2,5)$ still preserves a cooperative interaction $w_{2,5}=0.2484$.
Consequently, the delay-embedded signed Laplacian $\tiLgtau$ is PSD with a simple zero eigenvalue.
It implies that consensus can be achieved, as shown in Fig.~\ref{fig-delaymulti}.

\begin{figure}[!h]
  \centering
  \includegraphics[width=2.4in]{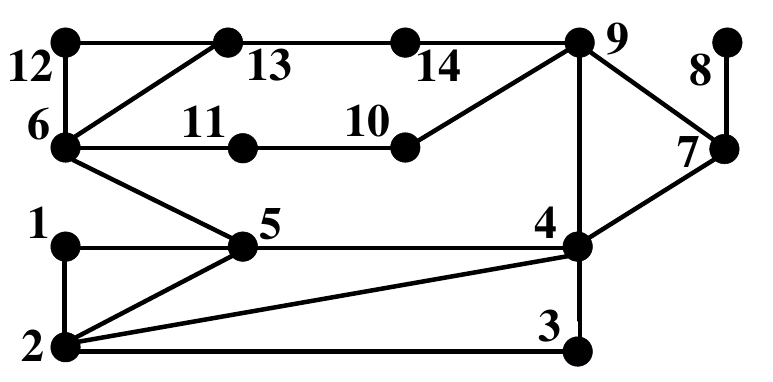}
  \caption{Network topology of the 14-node test system.}
  \label{fig-case14}
\end{figure}

%\begin{figure}[!h]
%  \centering
%  \includegraphics[width=3.0in]{2_5delay1000}
%  \caption{Consensus under an unweighted network and $\tau_{2,5}=1000$.}
%  \label{fig-delay1000}
%\end{figure}
%
%\begin{figure}[!h]
%  \centering
%  \includegraphics[width=3.0in]{multidelay2}
%  \caption{Consensus under an unweighted network and multiple delays $\tau_{2,5}=0.14$, $\tau_{6,13}=0.28$, $\tau_{2,4}=0.42$.}
%  \label{fig-delaymulti}
%\end{figure}

\begin{figure}[!h]
  \centering
  \subfigure[Single delay $\tau_{2,5}=1000$.]{
  \label{fig-delay1000} %% label for first subfigure
  \includegraphics[width=2.9in]{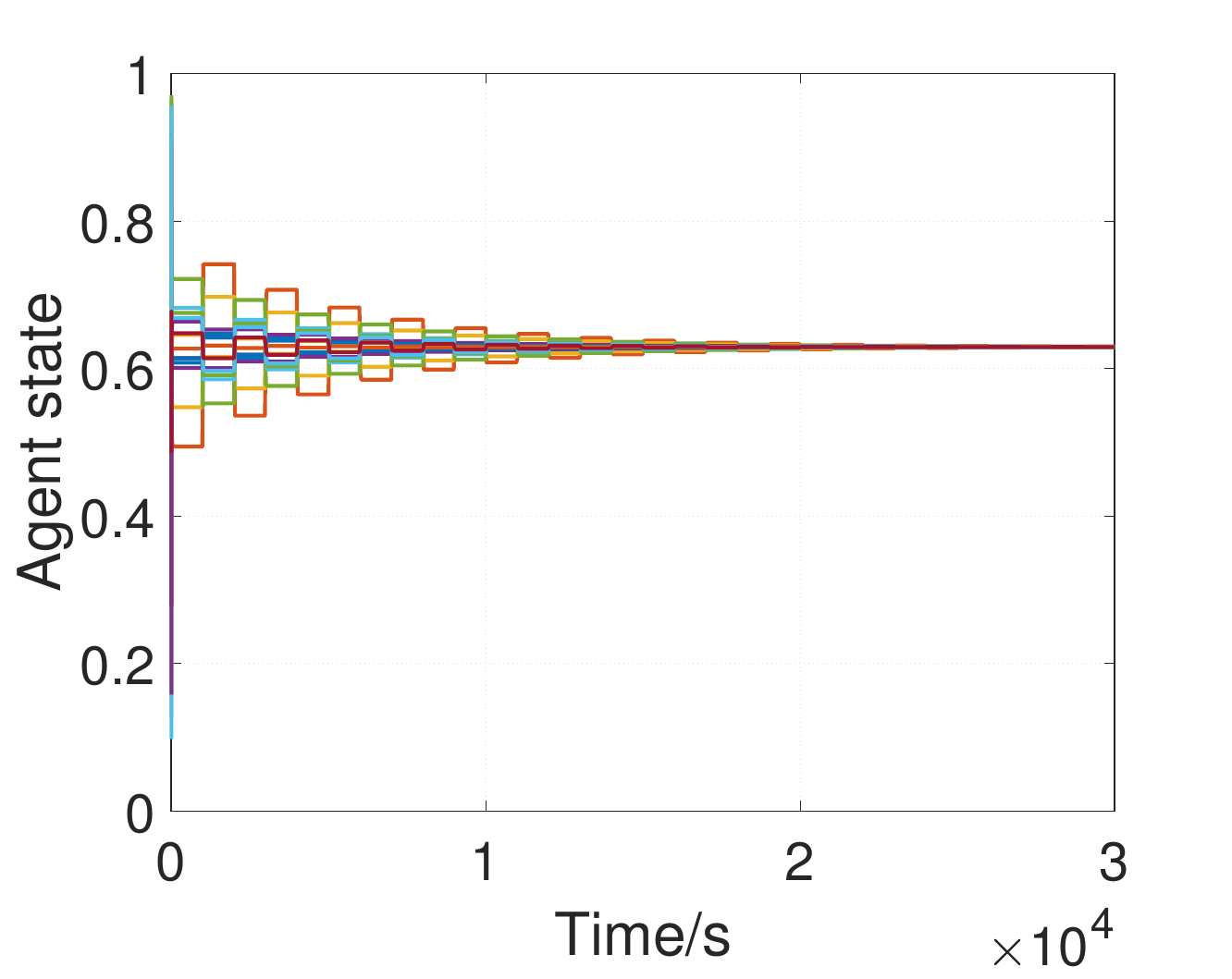}}
  \subfigure[multiple delays $\tau_{2,5}=0.14$, $\tau_{6,13}=0.28$, $\tau_{2,4}=0.42$.]{
  \label{fig-delaymulti} %% label for second subfigure
  \includegraphics[width=2.9in]{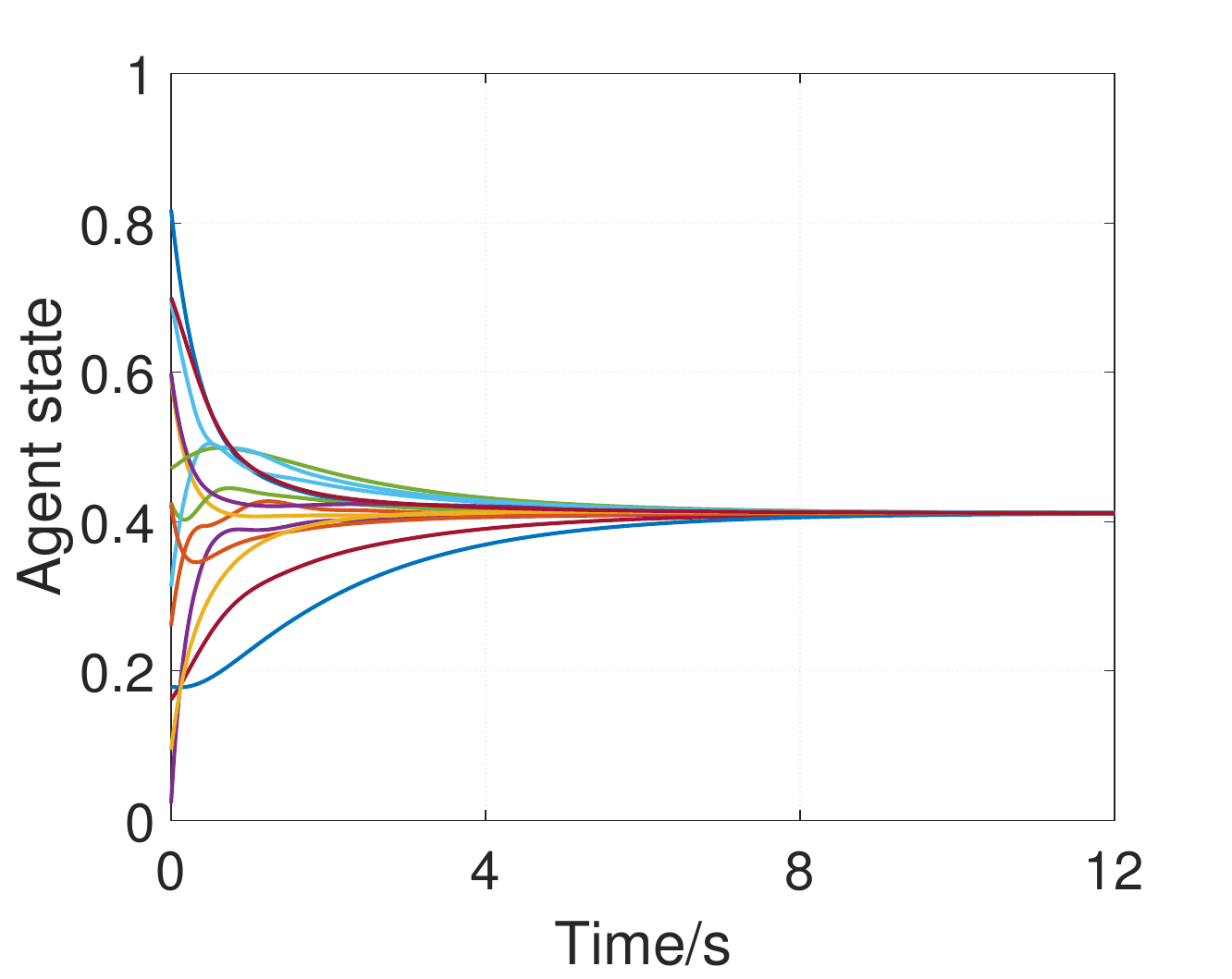}}
  \caption{Consensus simulation under delayed links.}
  \label{fig-delay} %% label for entire figure
\end{figure}

\subsection{Simple application in network design}
For the network with multiple delayed links described in the previous subsection, consider adding a link with $a_{kl}=0.77$ and $\tau_{kl}=10$, where $(k,l)$ is to be determined.
Treating the Laplacians before and after link addition as $\tiLgtau^{(k,l)}$ and $\tiLgtau$, respectively, we can apply Corollary \ref{cor-reff} to find a suitable place to add link $(k,l)$.
For instance, a direct calculation gives that $1/\reff(\tiLgtau^{(3,5)},\{3\},\{5\})=0.7772$ and $1/\reff(\tiLgtau^{(2,11)},\{2\},\{11\})=0.0812$, which means \eqref{singlenegative} is satisfied if $(k,l)=(3,5)$, and severely violated if $(k,l)=(2,11)$.
It implies that placing the new link at $(3,5)$ surely guarantees consensusability, and placing it at $(2,11)$ is highly risky, which is confirmed by the simulation results in Fig. \ref{fignewlink}.

\begin{figure}[!h]
  \centering
  \subfigure[After adding link $\tau_{3,5}=10$.]{
  \label{fig-35new} %% label for first subfigure
  \includegraphics[width=2.9in]{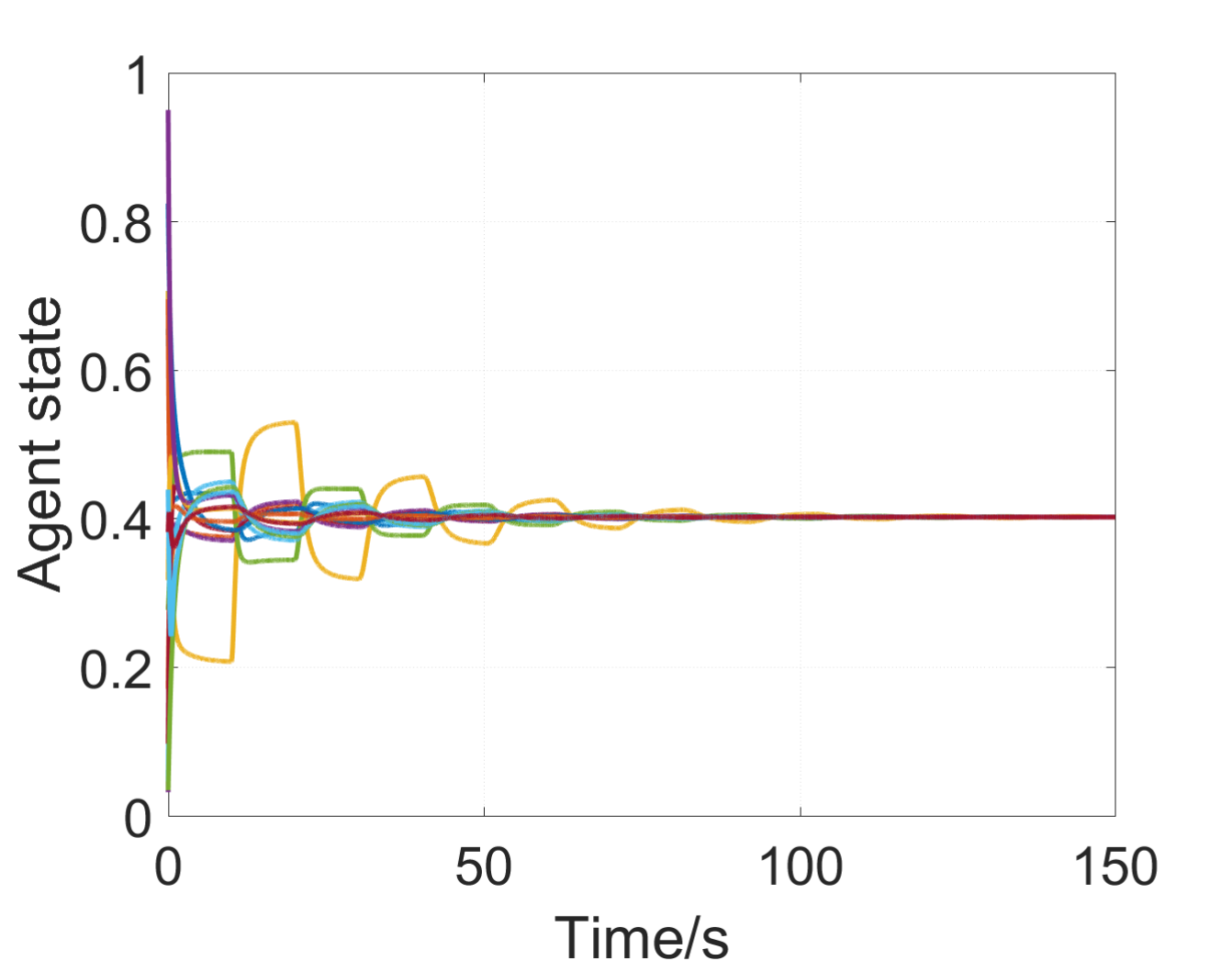}}
  \subfigure[After adding link $\tau_{2,11}=10$.]{
  \label{fig-211new} %% label for second subfigure
  \includegraphics[width=2.9in]{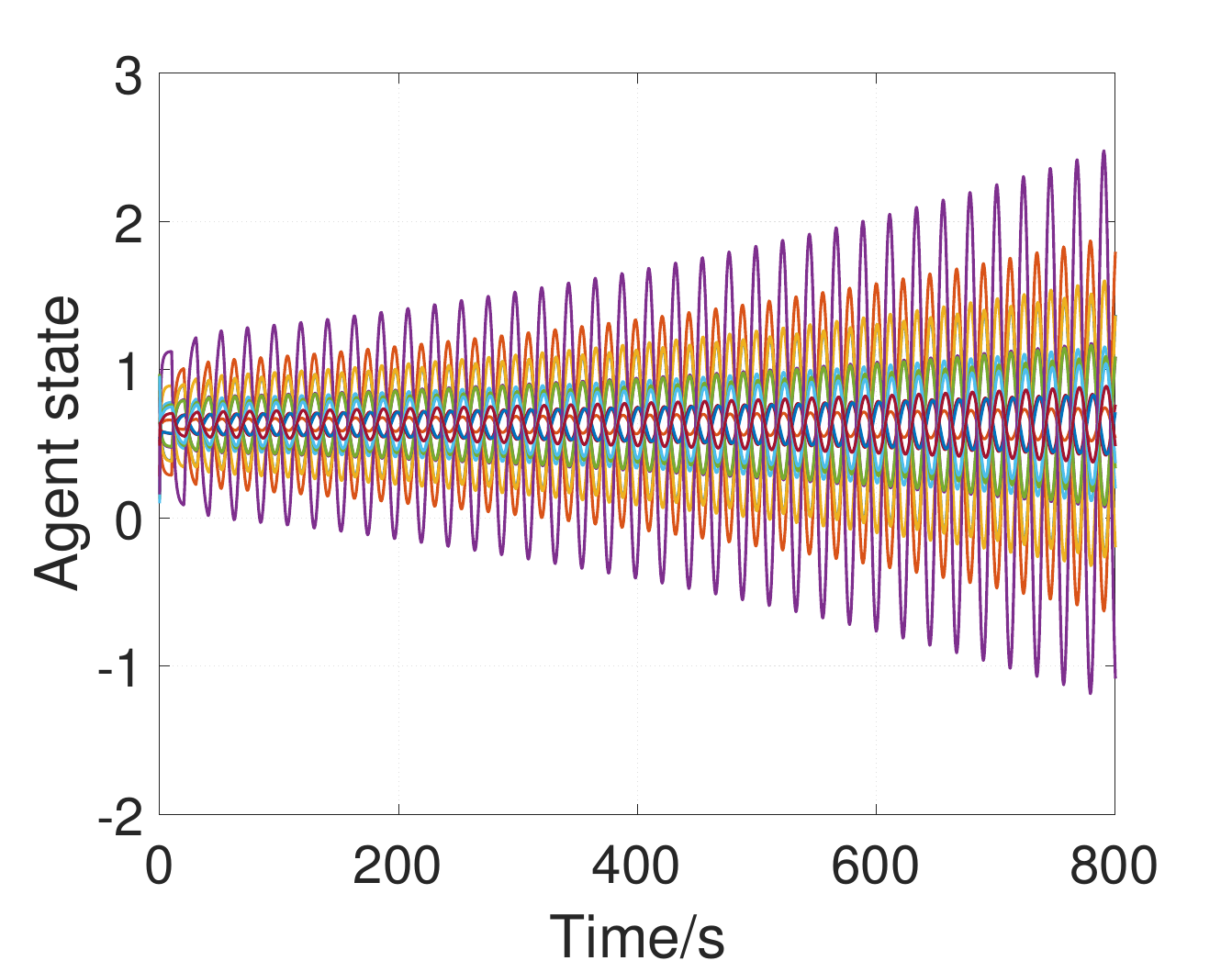}}
  \caption{Consensus simulation after link addition.}
  \label{fignewlink} %% label for entire figure
\end{figure}

\section{Conclusion}\label{secconclu}
The consensus with unbounded heterogeneous delays has been investigated by constructing several type of signed Laplacians carrying the delay information over the network topology.
The obtained results provide a network-oriented interpretation for the impact of large-delay links on consensusability.

\ifCLASSOPTIONcaptionsoff
  \newpage
\fi

{\footnotesize
\bibliographystyle{IEEEtran}
\bibliography{IEEEabrv,delay}

}

\end{document}